\documentclass[11pt]{amsart}
\usepackage[utf8]{inputenc}
\usepackage[linktocpage]{hyperref}
\usepackage{amssymb, paralist, xspace, graphicx, url, amscd, euscript, mathrsfs,stmaryrd,epic,eepic,color}
\usepackage[all]{xy}
\usepackage{amsthm}
\usepackage{enumerate}
\usepackage{multirow}
\usepackage{colortbl}
\usepackage{comment}
\definecolor{kugray5}{RGB}{224,224,224}
\usepackage{lipsum}
\usepackage{multirow}
\usepackage{tikz-cd}

\newtheorem{Theorem}{Theorem}
\newtheorem{theorem}[Theorem]{Theorem}
\newtheorem{proposition}[Theorem]{Proposition}

\newtheorem{corollary}[Theorem]{Corollary}

\newtheorem{lemma}[Theorem]{Lemma}
\theoremstyle{definition}
\newtheorem{definition}[Theorem]{Definition}

\newtheorem{remark}[Theorem]{Remark}

\newcommand\cF{\mathcal{F}}

\newcommand\cL{\mathcal{L}}

\newcommand\cT{\mathcal{T}}
\newcommand\cU{\mathcal{U}}

\newcommand\EE{\mathbb{E}}

\newcommand\HH{\mathbb{H}}

\newcommand\QQ{\mathbb{Q}}
\newcommand\RR{\mathbb{R}}

\newcommand\rD{\mathrm{D}}

\newcommand\rH{\mathrm{H}}

\newcommand\rR{\mathrm{R}}

\newcommand\NL{{\rm NL}}

\newcommand\rank{{\rm rank}}
\newcommand\Pic{{\rm Pic}}

\newcommand{\SO}{{\rm SO}}

\newcommand{\Hdg}{{\rm Hdg}}

\newcommand{\CH}{{\rm CH}}
\newcommand{\BV}{{\rm BV}}

\usetikzlibrary{arrows,shapes,chains}
\SelectTips{cm}{}

\title[Tautological classes of hyper-Kähler manifolds. {\it Erratum}]{Tautological classes on moduli spaces of hyper-Kähler manifolds. {\it Erratum}}

\author{Nicolas Bergeron} 

\address{Sorbonne Universit\'e, Institut de Math\'ematiques de Jussieu--Paris Rive Gauche, CNRS, Univ Paris Diderot, F-75005, Paris, France}

\email{nicolas.bergeron@imj-prg.fr}

\author{Zhiyuan Li}
\address{Shanghai Center for Mathematical Sciences \\ Fudan University\\
220 Handan Road, Shanghai, 200433 China\\
	}
\email{zhiyuan\_li@fudan.edu.cn}

\date{April 2021}

\begin{document}

\maketitle

\begin{abstract} This note is an erratum to the paper ``Tautological classes on moduli spaces of hyper-K\"ahler manifolds.'' Thorsten Beckman and Mirko Mauri have pointed to us a gap in the proof of \cite[Theorem 8.2.1]{Duke}. We do not know how to correct the proof. We can only recover a partial statement. This gap affects the proof of one of the two main results of \cite{Duke}, we explain how to correct it. 
\end{abstract}

\section{Notation}

Following notation of \cite[Section 3]{Duke} we let $\mathcal{F}_h^\ell$ be a connected component of the moduli space of $h$-polarised (or $h$-ample) hyper-K\"ahler manifolds of dimension $2n$ and with second Betti number $b_2=b+3$, with a full $\ell$-level structure. As in \cite[\S 3.7]{Duke} we furthermore denote by 
$\mathcal{F}_{\Sigma , h}^\ell$ the moduli space of $h$-ample $\Sigma$-polarised hyper-K\"ahler manifolds with full $\ell$-level structure. We have a natural forgetful map 
\begin{equation} \label{map}
\iota_{\Sigma} : \mathcal{F}_{\Sigma , h}^\ell \to \mathcal{F}_h^\ell.
\end{equation}

Let 
$$\pi_{h}^\ell : \mathcal{U}_h^\ell \to \mathcal{F}_h^\ell \quad \mbox{and} \quad \pi_{\Sigma, h}^\ell : \mathcal{U}_{\Sigma ,h}^\ell \to \mathcal{F}_{\Sigma , h}^\ell$$
be the corresponding universal families, let $r+1$ be the Picard number of the generic fiber of $\pi_h^\ell$, and let 
$$\mathcal{B}_\Sigma^\ell = \{ \mathcal{L}_0 , \ldots , \mathcal{L}_r \} \subset \mathrm{Pic}_\mathbb{Q} (\mathcal{U}_{\Sigma , h}^\ell )$$
be a collection of line bundles whose images in $\mathrm{Pic}_\mathbb{Q} (\mathcal{U}_{\Sigma , h}^\ell / \mathcal{F}_{\Sigma , h}^\ell )$ form a basis. 

We define the following subalgebras in $\mathrm{CH}^\bullet (\mathcal{F}_h^\ell )$:
\begin{itemize}
\item $\mathrm{NL}^\bullet  (\mathcal{F}_h^\ell )$ is the subalgebra generated by irreducible components of the images of the maps (\ref{map}) as one varies $\Sigma$;
\item the \emph{tautological ring} $\mathrm{R}^\bullet  (\mathcal{F}_h^\ell )$ is the subalgebra generated by the $\kappa$-cycles 
$$(\iota_{\Sigma} \circ \pi_{\Sigma, h}^\ell )_*  \left( \prod_{i=0}^r c_1 ( \mathcal{L}_i )^{a_i} \prod_{j=1}^{2n} c_j (T_{\pi_{\Sigma , h}^\ell})^{b_j} \right) ;$$
\item the \emph{special tautological ring} $\mathrm{DR}^\bullet (\mathcal{F}_h^\ell )$ is the subalgebra generated by the special $\kappa$-cycles 
$$(\iota_{\Sigma} \circ \pi_{\Sigma, h}^\ell )_*  \left( \prod_{i=0}^r c_1 ( \mathcal{L}_i )^{a_i}  \right).$$
\end{itemize}
We add the subscript \textit{hom} to denote the images of the the corresponding rings in $\rH^\bullet (\mathcal{F}_h^\ell , \mathbb{Q})$ via the cycle class map.  e.g. $\mathrm{CH}_{\rm hom}^k (\mathcal{F}_h^\ell )$ denotes the image of $\mathrm{CH}^k (\mathcal{F}_h^\ell )$ in $H^{2k} (\mathcal{F}_h^\ell , \mathbb{Q})$.

\begin{definition}
If $\mathcal{U} = \mathcal{U}_h^\ell$ or $\mathcal{U}_{\Sigma ,h}^\ell$, we will denote by $\rH^\bullet_{\leq d}( \mathcal{U} )$ the subalgebra of $\rH^\bullet(\mathcal{U},\QQ)$ generated by the classes in
$$\Hdg(\cU)^{2k}:=(W_{2k} \rH^{2k} (\mathcal{U} ))^{k,k} \quad \mbox{with } k \leq d.$$
when $d\leq 2n-1$. When $d=2n$, we define $\rH^\bullet_{\leq 2n}( \mathcal{U} )$ to be the subalgebra generated by $\rH^\bullet_{\leq 2n-1}( \mathcal{U} )$ and the relative Chern class $c_{2n}(\cT_\pi)$.  
\end{definition}

 Recall from \cite[\S 2.5]{Duke} that if $X$ is a hyper-K\"ahler manifold we denote by $\BV^\bullet(X)$ the subalgebra of $\mathrm{CH}^\bullet (X)$ generated by all divisors and Chern classes $c_i (T_X)$. Similarly we define the Beauville-Voisin ring 
$$\BV^\bullet_{\hom}(\cU)\subset \mathrm{CH}_{\rm hom}^\bullet (\mathcal{U})$$ 
to be the subalgebra generated by the cycle classes of $c_1 ( \mathcal{L}_i )$ and $ c_j (T_{\pi_{\Sigma , h}^\ell})$.  Note that  $\BV^\bullet_{\hom}(\cU)$ is a subring of $\rH^\bullet_{\leq 2n}(\cU)$. We finally define 
$$\widetilde{\rR}_{\hom}^\ast(\cF_h^\ell)\subseteq \rH^\bullet(\cF_h^\ell,\QQ)$$ 
to be the subring generated by all the pushforwards 
$$(\pi_{\Sigma, h}^\ell)_\ast( x) \quad \mbox{with} \quad x \in \rH^\bullet_{\leq 2n} (\mathcal{U}_{\Sigma ,h}^\ell),$$ 
where we let $\pi_{\Sigma, h}^\ell : \mathcal{U}_{\Sigma ,h}^\ell \to \mathcal{F}_{\Sigma , h}^\ell$ vary over all $\Sigma$. 
By definition, we have  
$$ \rD\rR^\bullet_{\hom}(\cF_h^\ell)\subseteq \rR^\ast_{\hom}(\cF_h^\ell)\subseteq \widetilde{\rR}_{\hom}(\cF_h^\ell).$$

\section{The Leray spectral sequence and cup products: a gap}

Following \cite[\S 8.2]{Duke} in this section we simply denote by $\pi: \cU\to \cF$ a universal family of lattice polarized hyperk\"ahler manifolds $\pi_{\Sigma, h}^\ell : \mathcal{U}_{\Sigma ,h}^\ell \to \mathcal{F}_{\Sigma , h}^\ell$. 

Consider the local systems $\mathbb{H}^j = R^j \pi_* \mathbb{Q}$. It follows from \cite[Theorem 1.1.1]{deCataldo} that there is a splitting of the Leray filtration in the category of mixed Hodge structure, the degeneration of the Leray spectral sequence for $\pi$ therefore gives an isomorphism of mixed Hodge structure 
\begin{equation} \label{E}
\rH^k (\mathcal{U} , \mathbb{Q} )  \cong \bigoplus_{i+j=k} \rH^i (\mathcal{F} , \mathbb{H}^j ).
\end{equation}
The cup product on the LHS of \eqref{E} induces a cup product on the Leray spectral sequence, and therefore a cup product on the RHS of \eqref{E}. However the isomorphism \eqref{E} does not preserve the ring structure; see e.g. \cite[Prop. 0.4 or Prop. 0.6]{Vo12}. 

In \cite[Theorem 8.2.1]{Duke}, we claimed that the two cup products on (\ref{E}) agree on the subalgebra  generated by low degree classes ($\mathrm{degree} < \frac12 \dim \cF$) up to removing some Noether-Lefschetz loci.  Based on this result, we proved  Theorem 8.3.1 and Theorem 4.3.1 in \cite{Duke}, which implies that
$$\rR_{\hom}^\bullet(\cF_h)=\NL^\bullet_{\hom}(\cF_h)$$
 when $n<\frac{b_2-3}{8}$. In particular, the cohomological tautological conjecture holds for K3 surfaces and $K3^{[2]}$-type hyperk\"ahler manifolds.  However, in our `proof', we only show that 

\begin{proposition} \label{P}
Let $\alpha_1$ and $\alpha_2$ be two classes in $\rH^\bullet (\mathcal{U},\QQ)$. Assume that the sum the degrees of $\alpha_1$ and $\alpha_2$ is $< \frac12 \dim \cF$, then 
\begin{equation}\label{E2}
\alpha_1 \wedge_{\rm LHS}  \alpha_2 = \alpha_1 \wedge_{\rm RHS}  \alpha_2,
\end{equation}
up to removing some Nother-Lefschetz components. 
\end{proposition}
\begin{proof}
This is a direct consequence of the cohomological generalized Franchetta Conjecture \cite[Theorem 8.1.1]{Duke}. 
\end{proof}
For the entire algebra $\rH^\bullet_{<\frac{1}{2}\dim \cF}(\cU)$, we actually do not know if \eqref{E2} holds (up to removing some NL loci). In other words, we don't know if the FKM ring is stable by the LHS ring structure in large degree.  

\section{The corrected cohomological tautological conjecture}

It remains open whether the two cup products coincides (up to some classes support on NL loci) on the entire subalgebra of $\CH^\bullet_{\rm hom} (\mathcal{U})$ generated by cycles of small co-dimension. This causes a problem in the proof of \cite[Theorem 8.3.1]{Duke} and hence of \cite[Theorem 4.3.1]{Duke}. 

As we will show in the next sections we can however prove the following theorem. The bound on $n$ is weaker but it is good enough to apply to $K3$ surfaces and $K3^{[2]}$-type hyper-K\"ahler manifolds. 

Recall that $\pi_{h}^\ell : \mathcal{U}_{h}^\ell \to \mathcal{F}_{h}^\ell$ denotes a smooth family of $h$-polarized hyperk\"ahler manifolds over an irreducible quasi-projective variety  $\cF_h^\ell$ of dimension $b \geq 3$. We define the following conditions  to state our main result:  

\medskip

\begin{quote}
{$\mathbf{(\ast )}$} \quad For every lattice-polarized universal family $\cU^\ell_{\Sigma,h}\to \cF^\ell_{\Sigma,h}$, the group of Hodge classes of {\bf degree $\mathbf{2n}$} of the very general fiber is spanned by Beauville-Voisin classes.    
\end{quote}
 
\medskip

\begin{quote}  
$\mathbf{(\ast \ast)}$ \quad   For every lattice-polarized universal family $\cU^\ell_{\Sigma,h}\to \cF^\ell_{\Sigma,h}$, the group of Hodge classes of {\bf degree $\mathbf{\leq 2n}$} of the very general fiber is spanned by Beauville-Voisin classes.
\end{quote}

\medskip

Then we have 
\begin{theorem}\label{tauthm}
Suppose  $b \geq \sup \{16n-12, 12n-4\}$. Then $$\mathrm{DR}^\bullet_{\hom} (\cF_h^\ell)=  \NL_{\hom}^\bullet(\cF_h^\ell).$$
Moreover, we have  $\rR^\bullet_{\hom} (\cF_h^\ell)=  \NL_{\hom}^\bullet(\cF_h^\ell)$ if  the condition $(\ast)$ holds and 
\begin{equation}\label{tauconj}
    \mathrm{NL}_{\hom}^\bullet (\mathcal{F}_h^\ell ) =\rR^\bullet_{\rm hom} (\mathcal{F}_h^\ell )= \widetilde{\rR}^\bullet_{\rm hom} (\mathcal{F}_h^\ell ),
\end{equation}
if either $b\geq \sup \{16n-8, 12n-2\}$ or the condition $(\ast\ast)$ holds. 
In particular, the cohomological tautological conjecture  holds for both $K3$ surfaces and $K3^{[2]}$-type hyper-K\"ahler manifolds. 
\end{theorem}

\section{Proof of Theorem \ref{tauthm}}

Before embarquing into the proof we first collect some results that quickly follow from theorems proved in \cite{Duke}. 

\subsection{Noether-Lefschetz locus}

\begin{definition}
Let $\alpha$ be a class in $\Hdg^{2k} (\cU^\ell_{\Sigma,h} )$. We say that $\alpha$ is \emph{supported on the Noether-Lefschetz locus} if there exist finitely many lattices $\Sigma_j \supset \Sigma$ with $\rank(\Sigma_j)=\rank (\Sigma)+1$ and classes $\gamma_j \in \Hdg^{2k-2} (\cU_{\Sigma_j ,h }^\ell )$ such that 
$$\alpha =\sum\limits_{j} r_j (\rho_j)_\ast (\gamma_j) ,$$
where $r_j \in \mathbb{Q}$ and $\rho_j$ denotes the map $\cU^\ell_{\Sigma_j ,h} \to \cU^\ell_{\Sigma,h}$ corresponding to the forgetful map $\mathcal{F}_{\Sigma_j , h}^\ell \to \mathcal{F}_{\Sigma , h}^\ell$. 
\end{definition}

Let $d = \dim \mathcal{F}_{\Sigma , h }^\ell$. The main results in \cite{BLMM16} and \cite{Duke} imply the following

\begin{theorem} \label{NLconj}
\begin{enumerate}
    \item [(i)] For every degree  $k<\frac{d+1}{3}$ or $k>\frac{2d-1}{3}$, we have 
     $$ \Hdg^{2k}(\cF_{\Sigma, h}^\ell )= \mathrm{NL}^{k}_{\rm hom}(\cF_{\Sigma, h}^\ell).$$
	\item [(ii)] Let $\alpha \in \Hdg^{2k} (\cU^\ell_{\Sigma,h} )$ with $k<\min \left\{ \frac{d}{4}+1 , \frac{d+1}{3} \right\}.$ 
If the restriction of $\alpha$ to the very general fiber of $\pi_{\Sigma,h}^\ell$ is zero, then $\alpha$ is supported on the Noether-Lefschetz locus. 
\end{enumerate}
\end{theorem}
\begin{proof} The first part is precisely the main result of \cite{BLMM16}.

The second assertion follows from \cite[Theorem 6.4.1]{Duke} and \cite[\S 6.2]{Duke}. In fact letting $E$ be a finite dimensional $\SO(2,d;\RR)$-representation and $\EE$ be the associated local system on $\cF_{\Sigma,h}^\ell$, it is proved there that 
\begin{quote}
$\left( \dagger \right)$ \quad \ for all $i<\frac{d}{4}$, the space $(W_{2i} H^{2i}(\cF_{\Sigma,h}^\ell , \mathbb{E}))^{i,i} $ is spanned \\ \phantom{espac} by decorated special cycles in $\mathrm{SC}^i (\cF_{\Sigma,h}^\ell , \mathbb{E})$ 
\end{quote} 
and, by definition, special cycle classes are generated by pushforwards of decorated fundamental classes of $\mathcal{F}_{\Sigma' , h}^\ell$, where $\Sigma'$ is a lattice containing $\Sigma$ and the pushforward is with respect to the forgetful map $\mathcal{F}_{\Sigma' , h}^\ell \to \mathcal{F}_{\Sigma , h}^\ell$. 

We now proceed as in the proof of \cite[Theorem 8.1.1]{Duke}: under the hypotheses of the theorem, the class  $\alpha$ belongs to 
$$\bigoplus_{i=1}^k (W_{2i} H^{2i}(\cF_{\Sigma,h}^\ell , \mathbb{E}))^{i,i}.$$
If $i < k$ then by hypothesis we have $i < \frac{d}{4}$ and $\left( \dagger \right)$ applied to $\mathbb{E} = \mathbb{H}^{2(k-i)}$ implies that the component of $\alpha$ in $(W_{2i} H^{2i}(\cF_{\Sigma,h}^\ell , \mathbb{E}))^{i,i} $ is a linear combination of decorated Noether-Lefschetz cycle classes. Finally if $i=k$ the local system $\mathbb{H}^{2(k-i)}$ is trivial and we can similarly apply (i) to conclude that all the components of $\alpha$ can be decomposed as linear combinations of decorated Noether-Lefschetz cycle classes. 

We are therefore reduced to proving that a decorated Noether-Lefschetz cycle class represents a class in $H^\bullet (\cU^\ell_{\Sigma,h} )$ that is supported on the Noether-Lefschetz locus. To do so consider a forgetful map
$$\iota_{\Sigma' , \Sigma} : \mathcal{F}_{\Sigma' , h}^\ell \to \mathcal{F}_{\Sigma , h}^\ell \quad (\mbox{with } \Sigma \subset \Sigma' \mbox{ and } \rank(\Sigma')=\rank (\Sigma)+i),$$
a parallel vector $\mathbf{v}$ in a local system $\mathbb{H}^{2a}$, and the corresponding cycle class 
$$(\iota_{\Sigma' , \Sigma})_* \left( [\mathcal{F}_{\Sigma' , h}^\ell ] \otimes \mathbf{v} \right) \in H^{2i} ( \mathcal{F}_{\Sigma , h}^\ell , \mathbb{H}^{2a} ) \subset H^{2(i+a)} (\cU^\ell_{\Sigma,h} ).$$
The parallel vector $\mathbf{v}$ corresponds to a class $\gamma$ in $\Hdg^{2a} (\cU^\ell_{\Sigma' ,h})$ whose restriction to the very general fiber is precisely $\mathbf{v}$. 

By induction on $d$ and up to classes supported on the Noether-Lefschetz locus of $\cU^\ell_{\Sigma' ,h}$ the class $\gamma$ is equal to  
$$ [\mathcal{F}_{\Sigma' , h}^\ell ] \otimes \mathbf{v} \in H^0 (\mathcal{F}_{\Sigma' , h}^\ell , \mathbb{H}^a )$$
in any given decomposition \eqref{E} of $H^\bullet (\cU^\ell_{\Sigma' ,h})$. Since $(\iota_{\Sigma' , \Sigma})_* \gamma$ is obviously supported on the Noether-Lefschetz locus, it follows that  
$$(\iota_{\Sigma' , \Sigma})_* \left( [\mathcal{F}_{\Sigma' , h}^\ell ] \otimes \mathbf{v} \right)$$
and therefore $\alpha$ are also supported on the Noether-Lefschetz locus. 
\end{proof}

Another key result is the following 
\begin{corollary}\label{NLpart}
Assume that $d < \frac{b}{2}+1$. Then for any $\alpha\in \CH^k(\cF_{\Sigma, h}^\ell )$, the cohomology class $(\iota_\Sigma )_\ast [\alpha]$ is lying in $\NL^{k+b-d}_{\hom} (\cF_h^\ell )$. 
\end{corollary}
\begin{proof}

By assumption, the class $[\alpha]$ belongs to $\Hdg^{2k} (\cF_{\Sigma , h}^\ell )$ and its pushforward image  $$(\iota_\Sigma )_\ast [\alpha] \in \Hdg^{2(k+b-d)}(\cF_h^\ell).$$
We then have distinguish two cases.
\begin{enumerate}
    \item  If $k+b-d>\frac{2b-1}{3}$,   it follows from Theorem \ref{NLconj} (i), applied to $\cF_h^\ell$, that $(\iota_\Sigma )_\ast [\alpha ]$ lies in $\NL_{\hom}^{k+b-d}(\cF_h^\ell)$.  
    \item  If  $k+b-d \leq \frac{2b-1}{3}$, then $k\leq d-\frac{b+1}{3} <\frac{d+1}{3}$ and it follows from Theorem \ref{NLconj} (i), applied to $\cF_{\Sigma , h}^\ell$, that $[\alpha] \in \NL^k_{\hom}(\cF_{\Sigma , h}^\ell )$. The assertion then follows from the fact that $(\iota_\Sigma )_\ast (\NL^\bullet (\cF_{\Sigma , h}^\ell ))\subseteq \NL^\bullet (\cF_h^\ell )$. 
\end{enumerate}
\end{proof}

\subsection{Inductive step}
Throughout this subsection, we simply denote by $\pi:\cU\to \cF$ the universal family $\cU_{\Sigma,h}^\ell\to \cF_{\Sigma,h}^\ell$. 
We make use of Theorem \ref{NLconj}(ii) to investigate the difference between the  rings $$\mathrm{DCH}^\bullet_{\hom}(\cU)\subseteq  \BV^\bullet_{\hom}(\cU)\subseteq \rH^\bullet_{\leq 2n}(\cU).$$ The following result shows that when $d=\dim \cF$ is large enough,  they only differ by some classes supported on the Noether-Lefschetz loci. 

 \begin{theorem}\label{induction1}
 Let $\alpha=\prod\limits_{i}\alpha_i \in \rH_{\leq 2n}^{2k}(\cU)$ with $k\geq 2n$, where each $\alpha_i$ belongs to $\Hdg^{2k_i} (\mathcal{U})$  with $k_i\leq 2n$ and  $\alpha_i=c_{2n}(T_\pi)$ if $k_i=2n$. Suppose one of the following conditions holds
 \begin{enumerate}
     \item $\dim \cF\geq \sup \{8n-3,6n \}$, or
     \item $\dim \cF\geq \sup \{8n-7 , 6n-3 \}$, $k_i<2n$ for all $i$, and  $k_{i_0}\neq n$ for some $i_0$.
 \end{enumerate} 
 Then, there exists $\beta \in \mathrm{DCH}^{2k}_{\hom}(\cU)$   such that 
 \begin{equation}\label{NLsupport}
     \alpha-\beta=\sum\limits_{j} r_j (\rho_j )_\ast (\gamma_j ),  
     \end{equation} with $r_j \in \mathbb{Q}$ and $\gamma_j \in H^{2k-2}_{\leq 2n-1}(\cU_{\Sigma_j })$ for some lattices  $\Sigma_j \supset \Sigma$ with $\rank(\Sigma_j)=\rank (\Sigma)+1$. 
     
     Moreover, in Case (1), $\beta$ can be chosen of the form $a\cL^{k}$ for some relative ample line bundle $\cL$, while in Case (2), $\beta$ can be chosen of the form $\cL^{k-1}\cL'$  for some $\cL,\cL'\in \Pic(\cU)$ with $\cL$ relative ample.
\end{theorem}
\begin{proof} Let 
$$\delta = \min \left\{\frac{d}{4}+1, \frac{d+1}{3} \right\}$$
be the constant appearing in Theorem \ref{NLconj}(ii). 

\begin{lemma} \label{LLL}
Let $\alpha \in \Hdg^{2k}(\cU)$ with $k\in ]n , 2n]$. Suppose furthermore $k < \delta$. Then, there exist $\cL\in \Pic(\cU)$ relative ample and $\beta \in \Hdg^{4n-2k} (\mathcal{U} )$ such that the difference $\alpha - \cL^{2k-2n} \beta$ is supported on the Noether-Lefschetz locus.
\end{lemma}
\begin{proof}
According to the relative Hard Lefschetz isomorphism 
$$H^0(\cF, \HH^{2k})\cong H^0(\cF, \HH^{4n-2k}),$$ 
we can find a class $\cL^{2k-2n}\beta$ with $\cL\in \Pic(\cU)$ relative ample and 
$$\beta\in   \Hdg^{4n-2k} (\mathcal{U} ) $$ 
such that the restriction of $\alpha - \cL^{2k-2n}\beta$ to each fiber is zero. Now by hypothesis we have $k<\delta$ and we can apply Theorem \ref{NLconj}(ii) to $\alpha -\cL^{2k-2n}\beta$. We conclude that the latter is supported on the Noether-Lefschetz locus of $\cF$.
\end{proof}

We will now make repeated use of this lemma to prove the theorem in each of the two cases.

First consider Case (1),  then by hypothesis  $2n<\delta$. Lemma \ref{LLL} applies to every product $\alpha_{i_1} \cdots \alpha_{i_r}$ with $k_{i_1} + \ldots + k_{i_r} \in ]n , 2n]$. Repeating this  we get a  class  $\beta' \in  \Hdg^{2m} (\mathcal{U})$ with  $m\leq n$ such that the difference $\alpha-\cL^{k-m} \beta'$ satisfies \eqref{NLsupport}. Note that the restrictions of $\cL^{2n-m}\beta'$ and $\cL^{2n}$ to the very general fiber are proportional, Lemma \ref{LLL} therefore implies that $\cL^{2n-m}\beta'-a\cL^{2n}$ is supported on the NL locus for some $a\in \QQ$. As a result,  we have  $$\cL^{k-m}\beta'-a\cL^{k}=\cL^{k-2n} (\cL^{2n-m}\beta'-a\cL^{2n} )$$ satisfies \eqref{NLsupport}. 

Now consider Case (2),  then by hypothesis $2n-1<\delta$ and there is at least one class $\alpha_{i_0}$ of degree $\neq n$. Applying Lemma \ref{LLL} 
as in Case (1) we therefore conclude that there exists a class $\beta'\in \Hdg^{2m}(\cU)$ with $m\leq n-1$ such that the difference 
$$\alpha-\cL^{k-m}\beta'\in \Hdg^{2k}(\cU)$$
satisfies \eqref{NLsupport}. By the relative Hard Lefschetz theorem, we can find a  line bundle $\cL'\in \Pic(\cU)$ such that $\cL^{2n-m-1}\beta'$ and $\cL^{2n-2}\cL'$ agree on the very general fiber. It follows from Theorem \ref{NLconj}(ii) that $\beta=\cL^{k-1}\cL'$ is as desired. 
\end{proof}

We derive the  following result directly  from Corollary \ref{NLpart} and Theorem~\ref{induction1}.

\begin{proposition}\label{DRR0}
Set $b=\dim \cF_h^\ell$. Suppose that $b \geq \sup \{16n-8 , 12n-2 \}$ and $\mathrm{DR}(\cF_h^\ell)=\NL_{\hom}^\bullet (\cF_h^\ell)$,
then we have
$$\mathrm{NL}_{\hom}^\bullet (\mathcal{F}_h^\ell ) =\rR^\bullet_{\rm hom} (\mathcal{F}_h^\ell )= \widetilde{\rR}^\bullet_{\rm hom} (\mathcal{F}_h^\ell ).$$

\end{proposition}
\begin{proof} 
To ease notation we keep using $\pi:\cU\to \cF$ to denote the universal family $\pi_{\Sigma}^\ell:\cU_{\Sigma,h}^\ell\to \cF_{\Sigma,h}^\ell$  and  let $\iota:\cF\to\cF_{h}^\ell$ be the forgetful map. 
Under the assumption of the proposition we will in fact prove (the stronger result) that 
\begin{equation}\label{NL-inclusion}
\iota_\ast(\pi_\ast \rH^\bullet_{\leq 2n}(\cU) ) \subseteq \NL^\bullet_{\hom}(\cF_h^\ell)    
\end{equation}
for all $\cU\to \cF$.

According to Theorem \ref{induction1}, once $\dim \cF\geq \sup \{8n-3 , 6n \}$,  any class in  $\pi_{\ast}(\rH^\bullet_{\leq 2n} (\cU))$ is a linear combination of classes in $\rD\rR^\bullet_{\hom}(\cF)$ and in the images $(\iota ')_\ast(\pi'_\ast(\rH^\bullet_{\leq 2n-1} (\cU')))$ as
$$\xymatrix{\cU'\ar[r]^{\rho'}\ar[d]^{\pi'} & \cU\ar[d]^{\pi} \\ \cF'\ar[r]^{\iota'} & \cF}$$ 
runs over all the universal family of  sublattice polarized hyper-K\"ahler varieties in $\cF$. 
By our assumption, the pushforward of $\rD\rR^\bullet_{\hom}(\cF)$ is lying in $\NL_{\hom}^\bullet(\cF_h^\ell)$. So it suffices to prove 
\begin{equation}
(\iota\circ\iota ' )_\ast(\pi'_\ast(\rH^\bullet_{\leq 2n-1} (\cU')))\subseteq \NL_{\hom}^\bullet(\cF_h^\ell),
\end{equation}
with $\dim \cF'<\dim \cF$.  
This allows us to  cut the dimension of $\cF$ whenever $\dim \cF\geq \sup \{8n-3 , 6n \}$ and we are reduced to prove that \eqref{NL-inclusion} holds as long as $\dim \cF \leq \sup \{8n-4 , 6n-1 \}$. This finally follows from Corollary \ref{NLpart} and our hypothesis that $b\geq \sup \{16n-8 , 12n-2 \}$ since the latter implies that $\dim \cF\leq \frac{1}{2}\dim \cF_h^\ell$.  
\end{proof}

Note that the bound of $b$ in Proposition \ref{DRR0} will be enough to prove Theorem \ref{tauthm} in the case of $K3$  surfaces as $b=19>10$, but it fails short to deal with the case of $K3^{[2]}$-type hyper-K\"ahler manifolds, where $b=20<24$. To strengthen Proposition \ref{DRR0}, we need the following result.

\begin{lemma} \label{L:T8refined}
Suppose $b\geq 16n-12$ and $n>1$.  Let $\pi:\cU\to \cF$ be the universal family of a lattice polarized hyper-K\"ahler varieties in $\cF_h^\ell$ with $\dim \cF\leq 8n-4$.  Then  for any class
$$\alpha=(c_{2n}(T_\pi))^m\prod  \alpha_i\in \rH^{2k}_{\leq 2n}(\cU)$$
where $\alpha_i\in \Hdg^{2k_i}(\cU)$ with $k_i<2n$, we have 
\begin{equation}\label{NLinclusion2}
 \iota_\ast(\pi_\ast(\alpha))\in \NL_{\hom}^\bullet(\cF_h^\ell)
 \end{equation}  if  one of the following  conditions holds
\begin{enumerate}
\item [(i)] $\dim \cF\leq 8n-6$;
    \item [(ii)] the $\alpha_i$'s are either Chern classes of $T_\pi$ or relative ample line bundles;
    \item [(iii)] there exists some $k_i\neq n$ and $\sum k_i\geq 2n$. 
   \end{enumerate}
In particular, this implies $\iota_\ast(\pi_\ast \BV^\bullet_{\hom}(\cU))\subseteq \NL^\bullet_{\hom}(\cF_h^\ell)$. 

\end{lemma}
\begin{proof}
As $b\geq 16n-12$, when the condition $(i)$ holds, we can conclude the assertion directly by Corollary \ref{NLpart}. It remains to consider the case when $\dim \cF=8n-5$ or $8n-4$.   We first consider the case where $\dim \cF = 8n-5 $. In that case, the proof of  Corollary \ref{NLpart} implies that \eqref{NLinclusion2} holds for every $k\neq \frac{8n-4}{3}+2n$. It remains to deal with the case where $k=\frac{8n-4}{3}+2n$ (the latter being forced to be an integer).  Write $\alpha' = \prod \alpha_i$; it is a class of degree 
$$\sum 2k_i=2k-4nm = \frac{1}{3} ( 4n(7-3m) -8 ).$$
With the assumption (ii) or (iii) holds, a simple exercise shows that we are in one of the following situations:
\begin{enumerate}
    \item [(a)] $m=2$ and $\alpha_i$ are relative Chern classes or relative ample line bundles, 
    \item [(b)] $m<2$ and there exists some $k_i\neq n$;
    \item [(c)] $n=2$ and $\alpha= [c_4 (T_\pi )] [c_2(T_{\pi})]^2$ or $[c_2(T_{\pi})]^4$.
\end{enumerate}
We now deal with each of these cases separately.

\subsubsection*{Case (a)} The class $\alpha'$ has degree $\frac{8n-16}{3}$ and it involves at most $\frac{2n-4}{3}$ distinct line bundles $\cL_i$. Let $\Sigma$ be the lattice generated by these line bundles and let $\cF''$ be the corresponding moduli space of $h$-ample $\Sigma$-polarized hyper-K\"ahler varieties. 
Then $\alpha'$ can be obtained as the pullback of a cohomology class $\tilde \alpha$ on the universal family above $\cF''$, in other words considering the diagram  
$$\xymatrix{\cU''\ar[d]^{\pi''} & \cU \ar[l]_{\rho''} \ar[d]^{\pi}\\ \cF'' & \cF\ar[l]_{\iota''}}$$
we have $\alpha' = (\rho'')^* \tilde \alpha$. 

Now we have $$\pi''_\ast(c_{2n}(T_\pi)^2 \tilde \alpha )\in \Hdg^{\frac{16n-8}{3}}(\cF'')$$ and 
$$\dim \cF''\geq \dim \cF_h^\ell-\frac{2n-4}{3}+1 > 8n-5 .$$
Theorem \ref{NLconj} therefore implies that  
$$\pi''_\ast(c_{2n}(T_\pi)^2 \tilde \alpha )\in \NL^{\bullet}_{\hom}(\cF'')$$ 
and hence 
$$\pi_\ast(\alpha)=(\iota'')^\ast( \pi''_\ast(c_{2n}(T_\pi)^2 \tilde \alpha))\in \NL^{\bullet}_{\hom}(\cF)$$ 
as pullback preserves the NL ring. 

\subsubsection*{In case (b)}  It follows from Theorem \ref{induction1} that $\alpha'-\cL^{k-2n-1}\cL'$ is supported on the NL loci of $\cF$ for some $\cL,\cL'\in \Pic(\cU)$. Then it suffices to show that 
\begin{equation}
\pi_\ast ([c_{2n}(T_\pi)]^m\cL^{k-2n-1}\cL')\in \NL^{k-2n}(\cF). 
\end{equation}
We proceed as in case (a). Let $\Sigma$ be the lattice generated by $\cL$ and $\cL'$. The class $[c_{2n}(T_\pi)]^m\cL^{k-2n-1}\cL'$ can be obtained as the pullback of a cohomology class on the universal family associated to the moduli space of $h$-ample $\Sigma$-polarized hyper-K\"ahler manifolds. The rest of the proof is similar. 

\subsubsection*{In case (c)} The class $\alpha$ can be obtained as the pullback of a class in  $\rH^{2k}_{\leq 2n}(\cU_h^\ell)$ and we proceed as in the first two cases. 

\medskip

Finally, in case $\dim \cF=8n-4$, \eqref{NLinclusion2} holds  for all $k\neq [\frac{8n-1}{3}]+2n$ from Corollary \ref{NLpart} and we can proceed the discussion for the case $k= [\frac{8n-1}{3}]+2n$ similarly. 
\end{proof}

\begin{proposition}\label{DRR}
Suppose that $b \geq 16n-12$, $n>1$ and $\mathrm{DR}^\bullet_{\hom}(\cF_h^\ell)=\NL_{\hom}^{\bullet}(\cF_h^\ell)$. Then we have  
$$\mathrm{NL}_{\hom}^\bullet (\mathcal{F}_h^\ell ) =\rR^\bullet_{\rm hom}(\mathcal{F}_h^\ell )$$
if $(\ast)$ holds;   
$$ \mathrm{NL}_{\hom}^\bullet (\mathcal{F}_h^\ell ) = \widetilde{\rR}^\bullet_{\rm hom} (\mathcal{F}_h^\ell ).$$
if  $(\ast\ast)$ holds. 
\end{proposition}
\begin{proof} 
According to the proof in Proposition \ref{DRR0}, we have 
$$ \mathrm{NL}_{\hom}^\bullet (\mathcal{F}_h^\ell ) = \widetilde{\rR}^\bullet_{\rm hom} (\mathcal{F}_h^\ell )$$
once the inclusion \eqref{NL-inclusion} holds for any  $\cU\to \cF$ with $\dim \cF\leq 8n-4$. So it suffices to check \eqref{NL-inclusion} when $\dim \cF\leq 8n-4$. When $(\ast\ast)$ holds, any class in $\rH^\bullet_{\leq 2n}(\cU)$ can be expressed as a linear combination of classes which, up to a class supported on the Noether-Lefschetz locus, satisfy the conditions in Lemma \ref{L:T8refined}. The assertion then follows from Lemma \ref{L:T8refined}.

Similarly, to prove 
$$ \mathrm{NL}_{\hom}^\bullet (\mathcal{F}_h^\ell ) = \rR^\bullet_{\rm hom} (\mathcal{F}_h^\ell ),$$
it suffices to show that 
\begin{equation}
\iota_\ast(\pi_\ast \rH^\bullet_{\leq 2n-1}(\cU) ) \subseteq \NL^\bullet_{\hom}(\cF_h^\ell)~ \hbox{and} ~\iota_\ast(\pi_\ast \BV^\bullet_{\hom} (\cU))\subseteq \NL^\bullet_{\hom}(\cF_h^\ell)
\end{equation}
for all $\cU\to \cF$ with $\dim \cF\leq 8n-4$. With the assumption $(\ast)$, any class in $\rH^\bullet_{\leq 2n-1}(\cU) $ can be written as a linear combination of classes which, up to a class supported on the NL locus, satisfy the conditions in Lemma \ref{L:T8refined}. Note that any class in $\rH^\bullet(\cU,\QQ)$ supporting on the NL locus of $\cF$ of dimension $\leq 8n-6$, its pushforward to $\rH^\bullet(\cF_h^\ell,\QQ)$  lies in $\NL_{\hom}(\cF_h^\ell)$.   Then our assertion  follows from Lemma \ref{L:T8refined}. 
\end{proof}

\subsection{Proof of Theorem \ref{tauthm}}

Let $b\geq \sup \{16n-12, 12n-4 \}$.  We first prove the following
\begin{lemma} \label{Lend1}
Let $\cL$ be a universal polarization of $\cU_h^\ell\to\cF_h^\ell$. For all integer $k \geq 2n+1$ we have 
\begin{equation}\label{DRNL}
    (\pi_h^\ell)_\ast (\cL^k)\in \NL^{k-2n}_{\hom}(\cF_h^\ell).
\end{equation}
\end{lemma}
\begin{proof}
We proceed by induction on $k \geq 2n+1$.  Theorem \ref{NLconj}(ii) implies that $\cL^{2n+1}$ is supported on the Noether-Lefschetz locus of $\cF_h^\ell$.  Assume \eqref{DRNL} holds for $k\leq k_0$. Now let $k=k_0+1>2n$. We are therefore reduced to prove that for every $\Sigma$ and every $\gamma\in \Hdg^{2j}(\cF_{\Sigma,h}^\ell)$ with $j\leq 2n-1$ and $j+j'+b=k+\dim \cF_{\Sigma,h}^\ell$, we have
\begin{equation}\label{inc-ind3}
(\iota_{\Sigma})_\ast(\pi_\Sigma)_\ast( \cL_\Sigma^{j'} \gamma)\in \NL^{k-2n}_{\hom}(\cF_{h}^\ell),
\end{equation}
where
$$\xymatrix{\cU_{\Sigma, h}^\ell \ar[r]^{\rho_\Sigma}\ar[d]^{\pi_\Sigma } & \cU_h^\ell \ar[d]^{\pi_h^\ell} \\ \cF_{\Sigma , h}^\ell \ar[r]^{\iota_\Sigma } & \cF_h^\ell}$$ 
and we write $\cL_{\Sigma}=\rho^\ast_\Sigma \cL$ for simplicity. 

The proof of \eqref{inc-ind3} is similar to that of Proposition \ref{DRR0}. We first explain how to reduce to $\cF_{\Sigma,h}^\ell$ with $\dim \cF_{\Sigma,h}\leq \sup \{8n-4 , 6n-1 \}$. Indeed: if $\dim \cF_{\Sigma,h}^\ell\geq \sup \{ 8n-3 , 6n \}$, Theorem \ref{induction1} implies that there exists a constant $a$ such that the difference 
\begin{equation}\label{ind-diff}
(\cL_\Sigma)^{2n-j} \gamma- a\cL_\Sigma^{2n}    
\end{equation}
is supported on the Noether-Lefschetz locus of $\cF_{\Sigma,h}^\ell$. It follows that 
\begin{equation}\label{diff-Nl-drr}
\cL_\Sigma^{j'} \gamma - a \cL_\Sigma^{k + \dim \cF_{\Sigma , h}^\ell -b} = \cL_\Sigma^{k+\dim \cF_{\Sigma,h}^\ell-b-2n} (\cL_\Sigma^{2n-j} \gamma- a \cL_\Sigma^{2n})
\end{equation}
can be expressed as linear combinations of the pushforward of the class $$\cL_{\Sigma_i}^{j_i} \gamma_i\in \rH^\bullet_{\leq 2n-1} (\cU^\ell_{\Sigma_i,h})$$ via $\rho_{\Sigma_i}$, where $\rho_{\Sigma_i}:\cU_{\Sigma_i,h}^\ell\to \cU_{\Sigma,h}^\ell$  is the universal family of some sublattice polarized hyper-K\"ahler varieties and $\gamma_i\in \Hdg^{2k_i}(\cF_{\Sigma_i,h}^\ell)$ with $k_i<2n$.

Since $b-\dim \cF_{\Sigma,h}^\ell \geq 1$, the inductive hypothesis implies that 
$$(\pi_h^\ell)_\ast\cL^{k+\dim \cF^\ell_{\Sigma,h}-b}\in \NL^{\bullet}_{\hom}{\cF_h^\ell},$$ 
therefore that 
\begin{equation}\label{eqind2}
a(\pi_{\Sigma})_\ast(\cL_\Sigma)^{k+\dim \cF_{\Sigma,h}^\ell-b}\in \NL^\bullet_{\hom}(\cF^\ell_{\Sigma,h}).
\end{equation}
and hence that the pushforward $(\iota_\Sigma)_\ast $ of the class \eqref{eqind2} is in $\NL^{k-2n}_{\hom}(\cF_h^\ell)$. 

We are therefore reduced  to show that $(\iota_{\Sigma_i})_\ast (\pi_{\Sigma_i})_\ast(\cL_{\Sigma_i}^{j_i} \gamma_i)\in\NL_{\hom}^\bullet(\cF_h^\ell)$ and this allows us to reduce the dimension of $\cF_{\Sigma,h}^\ell$. 

It remains to prove \eqref{inc-ind3} when $\dim \cF_{\Sigma,h}^\ell \leq \sup \{8n-4 , 6n-1 \}$. Since $b\geq \sup \{16n-12, 12n-4 \}$, Corollary \ref{NLpart} applies as long as $\dim \cF_{\Sigma,h}^\ell \leq \sup \{8n-6 , 6n-2 \}$ and we are left with the cases where $\dim \cF^\ell_{\Sigma,h}=8n-4$ or $8n-5$ (if $n>1$). 
In both cases,  any class of the form \eqref{inc-ind3} satisfies the condition (iii) of  Lemma \ref{L:T8refined} and hence belongs to  $\NL^\bullet_{\hom}(\cF_h^\ell)$. 

\end{proof}

Strengthening the proof of Lemma \ref{Lend1} we get:

\begin{lemma}\label{Lend3}
Consider  $\cF_{\Sigma,h}^\ell$ of dimension $\geq \sup \{8n-5 ,6n-2 \}$. Let $\cL_1 , \ldots , \cL_m \in  \Pic(\cU_{\Sigma,h}^{\ell} )$ and $\gamma\in   \Hdg^{2k} (\cU_{\Sigma,h}^{\ell})$ with $k<2n$. Suppose $m+k >2n$, then we have 
\begin{equation}\label{DRNL2}
(\iota_{\Sigma})_\ast (\pi_{\Sigma})_\ast(\cL_1\cL_2\ldots \cL_m \gamma)\in \NL^\bullet_{\hom}(\cF_{h}^\ell).
 \end{equation}
\end{lemma}
\begin{proof}
If $\dim \cF_{\Sigma,h}^\ell \geq \sup \{8n-3 ,6n \}$, Theorem \ref{induction1} implies that there exists a constant $a$ such that the class
\begin{equation}\label{difference}
    (\prod\limits_{i=1}^m \cL_i )\gamma- a \cL_\Sigma^{k+m}
\end{equation}
is supported on the Noether-Lefschetz locus of $\cF_{\Sigma,h}^\ell$. Here $\cL_\Sigma = \rho_\Sigma^\ast \cL$ still denotes the pull-back of the universal polarisation fixed in Lemma \ref{Lend1}.  

Since $m+k >2n$, Lemma \ref{Lend1} implies that $(\pi_h^\ell )_\ast (\cL^{k+m} )$ belongs to $\NL^\bullet_{\hom}(\cF_{h}^\ell)$. We conclude that 
$$(\pi_\Sigma)_\ast(\cL_\Sigma^{k+m}) \in \NL^\bullet_{\hom}(\cF_{\Sigma,h}^\ell).$$

It remains to prove that the pushforward of \eqref{difference} is lying in $\NL_{\hom}^\bullet (\cF_h^\ell)$. 
By a repeated use of Theorem \ref{induction1}, as in the proof of Lemma \ref{Lend1}, we are inductively reduced to prove that \eqref{DRNL2} holds when $\cF_{\Sigma,h}^\ell \leq \sup \{ 8n-4, 6n-1 \}$. The two cases we are left with follow from Lemma \ref{L:T8refined}.
\end{proof}

Lemma \ref{Lend3} implies in particular 
that for all $\cF_{\Sigma,h}^\ell$ of dimension $\geq \sup \{8n-3 , 6n \}$ we have
\begin{equation*}
(\iota_{\Sigma })_\ast((\pi_{\Sigma, h}^\ell)_\ast\mathrm{DCH}^\bullet_{\hom}(\cU_{\Sigma,h}^\ell))\subseteq \NL^\bullet_{\hom}(\cF_h^\ell),
\end{equation*} 
and hence $\mathrm{DR}_{\hom}^{\bullet}(\cF_h^\ell)=\NL_{\hom}^\bullet(\cF_h^\ell)$. Then we get \eqref{tauconj} from  Proposition \ref{DRR0} and  Proposition \ref{DRR}. 
 
 To conclude the proof of Theorem \ref{tauthm}, it remains to check that both $K3$ surfaces and $K3^{[2]}$-type hyper-K\"ahler manifolds satisfy the conditions. 

For $K3$ surfaces, the second Betti number is $22$ and hence $b=\dim \cF_h^\ell=19$, which is greater than both $16-8=8$ and $12-2= 10$.  

For $K3^{[2]}$-type hyper-K\"ahler manifolds, the second Betti number is $23$ and $b=\dim \cF_{h}^\ell=20$, which is exactly $16\times 2-12 = 12 \times 2 -4$.  In this case, as a representation of $\mathrm{O} (2,b ; \mathbb{R})$ we have $\HH^4=\mathrm{Sym}^2 (\HH^2)$. The primitive part of $\HH^2$ is the standard representation of $\mathrm{O}(h^\perp)$.  As a representation of $\mathrm{O}(\Sigma^\perp)$, the space $\HH^2$ decomposes as the direct sum of the subspace spanned by $\Sigma$ on which $\mathrm{O}(\Sigma^\perp)$ acts trivially and the standard representation $\HH^2_{\rm prim}$. 

Then the trivial isotypic subspace of $\mathrm{Sym}^2 (\HH^2)$ decomposes as the direct sum of $\mathrm{Sym}^2(\langle \Sigma \rangle )$ and the trivial summand in $\mathrm{Sym}^2(\HH^2_{prim})$; in otherwords the trivial isotypic subspace is spanned by products of line bundle classes and the second Chern class of the tangent bundle. This implies condition $(\ast )$. 

\begin{remark}
From our proof, one can see that we do not need the full strength of condition $(\ast)$ or $(\ast\ast)$. What we really need is that:  the group of Hodge classes of  degree $=2n$ (or $\leq 2n$ respectively) on general fibers of $\cU_{\Sigma,h}^\ell\to \cF_{\Sigma,h}^\ell$ is spanned by the product of line bundles and classes which can descends to the general fiber of $\cU_h^\ell\to \cF_h^\ell$.
\end{remark}

\section{Another correction}
There are some minor issues in the paper. In section 2.4, we say that the orthogonal group $G_X$ can act on the entire cohomology ring of $X$ via the monodromy action. This is not true is general. There is only a Spin group action on the entire cohomology ring and it does not factor through the orthogonal group in general (e.g. this fails for generalized Kummer varieties). But $G_X$ does act on the ring of even degree cohomology groups via the LLV action. Two actions agree after taking level structures (passing to a finite index subgroup).  See \cite[Remark 2.4.1]{Duke}. Note that the $G_X$ action is only needed to construct theta classes in even degrees, so this does not affect our results.

\bibliographystyle {plain}
\bibliography{main}

\def\cprime{$'$}
\begin{thebibliography}{1}

\bibitem{Duke}
Nicolas Bergeron and Zhiyuan Li.
\newblock Tautological classes on moduli spaces of hyper-{K}\"{a}hler
  manifolds.
\newblock {\em Duke Math. J.}, 168(7):1179--1230, 2019.

\bibitem{BLMM16}
Nicolas Bergeron, Zhiyuan Li, John Millson, and Colette Moeglin.
\newblock The {N}oether-{L}efschetz conjecture and generalizations.
\newblock {\em Invent. Math.}, 208(2):501--552, 2017.

\bibitem{deCataldo}
Mark Andrea~A. {de Cataldo}.
\newblock {Hodge-theoretic splitting mechanisms for projective maps}.
\newblock {\em {J. Singul.}}, 7:134--156, 2013.

\bibitem{Vo12}
Claire Voisin.
\newblock Chow rings and decomposition theorems for families of {$K3$} surfaces
  and {C}alabi-{Y}au hypersurfaces.
\newblock {\em Geom. Topol.}, 16(1):433--473, 2012.

\end{thebibliography}
\end{document}